\documentclass[a4paper,12pt]{amsart}

\usepackage[top=30truemm,bottom=30truemm,left=25truemm,right=25truemm]{geometry}

\usepackage{amsmath, amssymb}
\usepackage{amsthm}
\usepackage{amsfonts}
\usepackage{mathrsfs}
\usepackage{bm}
\usepackage{latexsym}

\newtheorem{thm}{Theorem}[section]
\newtheorem{lem}[thm]{Lemma}
\newtheorem{prop}[thm]{Proposition}

\theoremstyle{definition}
\newtheorem{prob}[thm]{Problem}

\newtheorem*{ex}{Example}

\DeclareMathOperator{\Res}{Res}

\newcommand{\abs}[1]{\left\lvert#1\right\rvert}

\makeatletter 
\@tempcnta\z@
\loop\ifnum\@tempcnta<26
\advance\@tempcnta\@ne
\expandafter\edef\csname f\@Alph\@tempcnta\endcsname{\noexpand\mathfrak{\@Alph\@tempcnta}}
\expandafter\edef\csname l\@Alph\@tempcnta\endcsname{\noexpand\mathbb{\@Alph\@tempcnta}}
\expandafter\edef\csname c\@Alph\@tempcnta\endcsname{\noexpand\mathcal{\@Alph\@tempcnta}}
\repeat

\title{The existence of $T$-numbers in positive characteristic}
\author{Tomohiro Ooto}
\date{}
\begin{document}
\address{Faculty of Pure and Applied Sciences, University of Tsukuba, Tennodai 1-1-1, Tsukuba, Ibaraki, 305-8571, Japan}
\email{ooto@math.tsukuba.ac.jp}

\begin{abstract}
	As an analogue of Mahler's classification for real numbers, Bundschuh introduced a classification for Laurent series over a finite field, divided into $A,S,T,U$-numbers.
	It is known that each of $A,S,U$-numbers is nonempty.
	On the other hand, the existence of $T$-numbers is open.
	In this paper, we give an affirmative answer to the problem.
\end{abstract}

\maketitle

\section{Introduction}\label{sec:intro}

From the viewpoint of Diophantine approximation, Mahler \cite{Mahler32} introduced a classification for real numbers, divided into $A$, $S$, $T$, $U$-numbers.
A real number is algebraic over $\lQ$ if and only if it is an $A$-number.
Two algebraically dependent transcendental real numbers are in the same class.
Almost all real numbers are $S$-numbers in the sense of Lebesgue measure.
For example, $e$ is $S$-number \cite{Popken29}.
Liouville numbers are $U$-numbers, for example, the real number $\sum _{n=1}^{\infty }1/2^{n!}$ is $U$-number.
Therefore, $e$ and $\sum _{n=1}^{\infty }1/2^{n!}$ are algebraically independent.
The existence of $T$-numbers had been open for thirty-six years.
Schmidt \cite{Schmidt69} proved that there exist uncountably many $T$-numbers.
The proof of this result is based on a nested interval construction and a generalization of the Roth's theorem by Wirsing \cite{Wirsing71}.
Since Wirsing's theorem is of ineffective nature, we note that Schmidt's construction does not give explicit examples of $T$-numbers.
After that, Schmidt \cite{Schmidt71} and Baker \cite{Baker76} investigated $T$-numbers in more detail.
We refer the reader to \cite{Bugeaud04} for 
Mahler's classification.

Let $p$ be a prime.
Mahler \cite{Mahler34} also introduced an analogue of Mahler's classification for $p$-adic numbers.
Schlickewei \cite{Schlickewei81} proved the existence of $p$-adic $T$-numbers by using a $p$-adic analogue of Schmidt's method \cite{Schmidt69}.
Pejkovi\'{c} \cite{Pejkovic13} investigated $p$-adic $T$-numbers in more detail.

Let $q$ be a power of $p$.
We denote by $\lF _q$ the finite field of $q$ elements, $\lF _q[T]$ the ring of all polynomials over $\lF _q$, $\lF _q(T)$ the field of all rational functions over $\lF _q$, and $\lF _q((T^{-1}))$ the field of all Laurent series over $\lF _q$.
We call an element of $\lF_q((T^{-1}))$ an \textit{algebraic} (resp.\ a \textit{transcendental}) \textit{Laurent series} if the element is algebraic (resp. transcendental) over $\lF_q(T)$.
Analogues to $\lZ, \lQ, $ and $\lR$ are $\lF_q[T], \lF_q(T), $ and $\lF_q((T^{-1}))$, respectively.
Bundschuh \cite{Bundschuh78} introduced an analogue of Mahler's classification in $\lF_q((T^{-1}))$.
As in the real case, he divided Laurent series into $A$, $S$, $T$, $U$-numbers (see Section \ref{sec:mainresult} for the precise definition).
A fundamental question for this classification is whether or not each of the classes is nonempty.
A Laurent series is algebraic over $\lF_q(T)$ if and only if it is an $A$-number (see Proposition \ref{prop:A-number}).
Two algebraically dependent transcendental Laurent series are in the same class (see Proposition \ref{prop:dependent}).
Sprind\u{z}uk \cite{Sprindzuk69} proved that almost all Laurent series are $S$-numbers in the sense of Haar measure.
It is easily seen that there exist $U$-numbers, for example, the Laurent series $\sum _{n=1}^{\infty }T^{-n!}$ is $U$-number.
Therefore, each of the classes except for $T$-numbers is known to be nonempty.
However, the existence of $T$-numbers in $\lF_q((T^{-1}))$ is open (see \cite{Dubois96, Thakur12}).

In this paper, we solve this open problem.

\begin{thm}\label{thm:intro}
	There exist uncountably many $T$-numbers in $\lF_q((T^{-1}))$.
\end{thm}

We emphasize that our method allows us to construct explicit examples of $T$-numbers in $\lF_q((T^{-1}))$.
For example, we define a sequence $(a_n)_{n\geq 0}$ over $\lF_2$ by
\begin{equation*}
a_n=\begin{cases}
1 & \text{if } n=2^{4^k\ell } \text{ for some integer } k\geq 0 \text{ and odd integer } \ell \geq 1,\\
0 & \text{otherwise}.
\end{cases}
\end{equation*}
Then the Laurent series $\sum_{n=0}^{\infty }a_nT^{-n}\in \lF_2((T^{-1}))$ is $T$-number.

In the field of Laurent series over a finite field, Mahler \cite{Mahler49} proved that an analogue of Roth's theorem does not hold and constructed a family of explicit counterexamples.
Therefore, we prove Theorem \ref{thm:intro} in a different way to Schmidt's proof.
Our strategy for proving Theorem \ref{thm:intro} is that we construct a Laurent series with Mahler's counterexamples and show that the Laurent series is $T$-number by using a Liouville inequality.

This paper is organized as follows.
In Section \ref{sec:mainresult}, we recall the precise definition of Mahler's classification in $\lF_q((T^{-1}))$ and another classification which is called Koksma's classification.
We also state the main results in this paper.
In Section \ref{sec:pre}, we prepare some lemmas for the proof of the main result.
In Section \ref{sec:proof}, we prove the main results and Theorem \ref{thm:intro}. 
In Appendix \ref{sec:app}, we prove basic properties of Mahler's classification stated in Section \ref{sec:intro} and \ref{sec:mainresult}.

\section{Notation and Main result}\label{sec:mainresult}

For a Laurent series $\xi \in \lF_q((T^{-1})) \setminus \{ 0\}$, we can write $\xi = \sum _{n=N}^{\infty }a_n T^{-n},$ where $N\in \lZ $, $a_n \in \lF _q$ for all integers $n\geq N$, and $a_N \neq 0$.
We define an absolute value on $\lF_q((T^{-1}))$ by $|0|:=0$ and $|\xi |:=q^{-N}$.
The absolute value can be uniquely extended on the algebraic closure of $\lF_q((T^{-1}))$.
We continue to write $|\cdot |$ for the extended absolute value.
We denote by $(\lF_q[T])[X]$ the set of all polynomials in $X$ over $\lF_q[T]$.
The \textit{height} of a polynomial $P(X)\in (\lF_q[T])[X]$, denoted by $H(P)$, is defined to be the maximum of the absolute values of the coefficients of $P(X)$.
We denote by $(\lF_q[T])[X]_{\min }$ the set of all non-constant, irreducible, and primitive polynomials in $(\lF_q[T])[X]$ whose leading coefficients are monic polynomials in $T$.
For an algebraic number $\alpha \in \overline{\lF_q(T)}$, there exists a unique polynomial $P(X)\in (\lF_q[T])[X]_{\min }$ such that $P(\alpha )=0$. 
We call the polynomial $P(X)$ the \textit{minimal polynomial} of $\alpha $.
The \textit{height} (resp.\ the \textit{degree}) of $\alpha $, denoted by $H(\alpha )$ (resp.\ $\deg \alpha $), is defined to be the height of $P(X)$ (resp.\ the degree of $P(X)$).

Let $n\geq 1$ be an integer and $\xi $ be in $\lF_q((T^{-1}))$.
We denote by $w_n(\xi )$ (resp.\ $w_n^{*}(\xi )$) the supremum of the real numbers $w$ (resp.\ $w^{*}$) which satisfy 
\begin{equation*}
	0
	<|P(\xi )|
	\leq H(P)^{-w}
	\quad (\text{resp.\ } 
	0
	<|\xi -\alpha |
	\leq H(\alpha )^{-w^{*}-1})
\end{equation*}
for infinitely many $P(X)\in (\lF_q[T])[X]$ of degree at most $n$ (resp.\ $\alpha \in \overline{\lF_q(T)}$ of degree at most $n$).
We put
\begin{equation*}
	w(\xi ) 
	:=\limsup _{n\rightarrow \infty }\frac{w_n(\xi )}{n},\quad 
	w^{*}(\xi ) 
	:=\limsup _{n\rightarrow \infty }\frac{w_n ^{*}(\xi )}{n}.
\end{equation*}
We say that a Laurent series $\xi \in \lF_q((T^{-1}))$ is an
\begin{gather*}
	A\text{-\textit{number} if } w(\xi )=0;\\
	S\text{-\textit{number} if } 0<w(\xi )<+\infty ;\\
	T\text{-\textit{number} if } w(\xi )=+\infty \text{ and } w_n(\xi )<+\infty \text{ for all integers } n\geq 1;\\
	U\text{-\textit{number} if } w(\xi )=+\infty \text{ and } w_n(\xi )=+\infty \text{ for some integer } n\geq 1.
\end{gather*}
This classification is called \textit{Mahler's classification}.
Replacing $w_n$ and $w$ with $w_n ^{*}$ and $w^{*}$, we define $A^{*}$, $S^{*}$, $T^{*}$, and $U^{*}$-numbers as the above.
This classification was first introduced by Bugeaud \cite[Section 9]{Bugeaud04} and is called \textit{Koksma's classification}.
$\xi $ is an $A$-number if and only if it is an $A^{*}$-number (see Proposition \ref{prop:A-number}).
The following two statements are in \cite[p.145]{Ooto17}.
If $\xi $ is an $S$-number, then it is an $S^{*}$-number.
$\xi $ is a $U$-number if and only if it is a $U^{*}$-number.
Therefore, we deduce that if $\xi $ is a $T^{*}$-number, then it is a $T$-number.

Let $\xi \in \lF_q((T^{-1}))$ be a $T$-number.
The \textit{type} of $\xi $, denoted by $t(\xi )$, is defined to be
\begin{equation*}
	t(\xi )
	= \limsup_{n\rightarrow \infty }\frac{\log w_n(\xi )}{\log n}.
\end{equation*}
Note that we see $t(\xi )\in [1,+\infty ]$ by Lemma \ref{lem:MKlower}.
Replacing $w_n$ with $w_n ^{*}$, for $T^{*}$-number $\xi \in \lF_q((T^{-1}))$, we define the $*$-type of $\xi $, denoted by $t^{*}(\xi )$ as the above.
Note that we also see $t^{*}(\xi )\in [1,+\infty ]$ by Lemma \ref{lem:MKlower}.

Let $r$ be a power of $p$.
We put $\alpha := \sum_{n=1}^{\infty }T^{-r^n}.$
Mahler \cite{Mahler49} showed that $\alpha $ is the algebraic Laurent series of degree $r$, and satisfies $w_1(\alpha )=r-1$ and 
\begin{equation}\label{eq:equation}
\alpha ^r = \alpha -T^{-r}. 
\end{equation}
Note that, in the case of $r>2$, the algebraic Laurent series $\alpha $ is the first counterexample of the Roth's theorem in $\lF_q((T^{-1}))$, that is, $\alpha $ does not satisfy $w_1(\alpha )=1$.

Let ${\bf m}=(m_j)_{j\geq 0}$ be an integer sequence with $m_0=1, m_j\geq 2$ for all integers $j\geq 1$.
For an integer $j\geq 0$, we put $r_j:=r^{m_0m_1\cdots m_j}$ and $\alpha _j(r,{\bf m}):=\sum_{n=1}^{\infty } T^{-r_j ^n}.$
We define a Laurent series $\xi (r,{\bf m})$ by
\begin{equation*}
	\xi (r,{\bf m}) 
	= \sum_{j=0}^{\infty }\alpha _j(r,{\bf m}).
\end{equation*}
Note that since $\lim_{j\rightarrow \infty }\abs{\alpha _j(r,{\bf m})} = 0$, the Laurent series $\xi (r,{\bf m})$ converges.

\begin{ex}
	Let $p=r=2$ and $m_j=2$ for all $j\geq 1$.
	We write $\xi (r,{\bf m})=\sum_{n=0}^{\infty }a_nT^{-n}$.
	Then we have
	\begin{equation*}
	a_n=\begin{cases}
	1 & \text{if } n=2^{4^k\ell } \text{ for some integer } k\geq 0 \text{ and odd integer } \ell \geq 1,\\
	0 & \text{otherwise}.
	\end{cases}
	\end{equation*}
\end{ex}

Our main result of this paper is the following theorem.

\begin{thm}\label{thm:main}
	For any $r$ ans ${\bf m}$ defined as above, the Laurent series $\xi (r,{\bf m})$ are $T$-numbers and $T^{*}$-numbers.
\end{thm}

We estimate type and $*$-type of the Laurent series $\xi (r,{\bf m})$.

\begin{thm}\label{thm:main2}
	For any $r$ ans ${\bf m}$ defined as above, the Laurent series $\xi (r,{\bf m})$ satisfies
	\begin{gather}\label{eq:main1}
		\limsup_{j\rightarrow \infty }(2m_j-1)
		\leq t^{*}(\xi (r,{\bf m}))
		\leq \limsup_{j\rightarrow \infty }(m_j+2m_jm_{j+1}\cdots m_{j+p}),\\
		\limsup_{j\rightarrow \infty }(2m_j-1)
		\leq t(\xi (r,{\bf m}))
		\leq \limsup_{j\rightarrow \infty }(2m_j+2m_jm_{j+1}\cdots m_{j+p}).\label{eq:main2}
	\end{gather}
	Furthermore, if $m_j\geq 3$ for all sufficiently large $j\geq 1$, then we have
	\begin{gather}\label{eq:main3}
		t^{*}(\xi (r,{\bf m}))
		\leq \limsup_{j\rightarrow \infty }(m_j+2m_jm_{j+1}),\\
		t(\xi (r,{\bf m})) 
		\leq \limsup_{j\rightarrow \infty }(2m_j+2m_jm_{j+1}).\label{eq:main4}
	\end{gather}
\end{thm}

In the last part of this section, we mention a problem concerning Theorem \ref{thm:main2}.

\begin{prob}\label{prob:anyvalue}
	For any $t\in [1,\infty ]$, does there exist a $T$-number $\xi $ (resp.\ $T^{*}$-number $\eta $) such that $t(\xi )=t$ (resp.\ $t^{*}(\eta )=t$)?
\end{prob}

If $\limsup_{j\rightarrow \infty }m_j = \infty $, then we have $t(\xi (r,{\bf m}))=t^{*}(\xi (r,{\bf m}))=\infty $ by Theorem \ref{thm:main2}.
Therefore, Theorem \ref{thm:main2} gives a partial answer to Problem \ref{prob:anyvalue} in the case of $t=\infty $.

\section{Preliminaries}\label{sec:pre}

\begin{lem}\label{lem:hikaku}
	Let $n\geq 1$ be an integer and $\xi $ be in $\lF_q((T^{-1}))$.
	Let $k\geq 0$ be an integer with $p^k\leq n<p^{k+1}$.
	Then we have
	\begin{equation*}
		\frac{w_n(\xi )}{p^k}-n+\frac{2}{p^k}-1
		\leq w_n^{*}(\xi )
		\leq w_n(\xi ).
	\end{equation*}
\end{lem}

\begin{proof}
	See Proposition 5.6 in \cite{Ooto17}.
\end{proof}

The following lemma is easy consequence of Lemma \ref{lem:hikaku}.

\begin{lem}\label{lem:tnumber}
	For a $T$-number and $T^{*}$-number $\xi \in \lF_q((T^{-1}))$, we have $t^{*}(\xi )\leq t(\xi )$.
\end{lem}

The following lemma is well-known and immediately seen.

\begin{lem}\label{lem:height}
	For $P(X),Q(X)\in (\lF_q[T])[X]$, we have $H(PQ) = H(P)H(Q)$.
\end{lem}

We recall a Liouville inequality which is Korollar 3 in \cite{Muller93} or Proposition 3.4 in \cite{Ooto17}.

\begin{lem}\label{lem:Liovilleineq}
	Let $\alpha ,\beta \in \overline{\lF_q(T)}$ be distinct algebraic numbers of degree $m$ and $n$, respectively.
	Then we have
	\begin{equation*}
		\abs{\alpha -\beta }
		\geq H(\alpha )^{-n} H(\beta )^{-m}.
	\end{equation*}
\end{lem}

As an application of the Liouville inequality, we show the following lemma.
Lemma \ref{lem:applio} means that if $\xi \in \lF_q((T^{-1}))$ has a dense (in a suitable sense) sequence of very good algebraic approximations of degree at most $d$, then we can estimate the upper bound of $w_d^{*}(\xi )$.
Some results which are relate to the lemma are known (see e.g. \cite{Adamczewski10, Amou91, Bugeaud12, Firicel13, Ooto17, Ooto172}).

\begin{lem}\label{lem:applio}
	Let $\xi $ be in $\lF_q((T^{-1}))$, $d\geq 1$ be an integer, and $\theta ,\rho ,\delta $ be positive numbers.
	Assume that there exists a sequence of distinct terms $(\alpha _j)_{j\geq 1}$, and an increasing and divergent sequence of real numbers $(\beta _j)_{j\geq 1}$ with $\alpha _j\in \overline{\lF_q(T)}$ of degree at most $d$ and $\beta _j\geq 1$ for all integers $j\geq 1$, such that
	\begin{gather*}
	d+\delta 
	\leq \liminf_{j\rightarrow \infty } \frac{-\log \abs{\xi -\alpha _j}}{\log \beta _j},\quad 
	\limsup_{j\rightarrow \infty } \frac{-\log \abs{\xi -\alpha _j}}{\log \beta _j}
	\leq d+\rho ,\\
	\limsup_{j\rightarrow \infty } \frac{\log \beta _{j+1}}{\log \beta _j}
	\leq \theta ,\quad 
	\limsup_{j\rightarrow \infty } \frac{\log H(\alpha _j)}{\log \beta _j}
	\leq 1. 
	\end{gather*}
	Then we have
	\begin{equation*}
		d+\delta -1
		\leq w_d ^{*}(\xi )
		\leq (d+\rho )\frac{d\theta }{\delta }-1.
	\end{equation*}
\end{lem}

\begin{proof}
	Let $\varepsilon $ be a positive number with $\varepsilon(1+d)<\delta$.
	Then, by the assumption, there exists an integer $c_0\geq 1$ such that
	\begin{equation*}
		\beta _j^{-d-\rho -\varepsilon }
		\leq \abs{\xi -\alpha _j}
		\leq \beta _j^{-d-\delta +\varepsilon },\quad
		\beta _j
		\leq
		\beta _{j+1}\leq
		\beta _j^{\theta +\varepsilon },\quad 
		H(\alpha _j)
		\leq \beta _j^{1+\varepsilon }
	\end{equation*} 
	for all integers $j\geq c_0$.
	Since $\varepsilon $ is arbitrary small, we obtain $d+\delta -1\leq w_d ^{*}(\xi )$.
	
	Let $\alpha \in \overline{\lF_q(T)}$ be an algebraic number of degree at most $d$ with sufficiently large height.
	We define an integer $j_0\geq c_0$ by
	\begin{equation*}
		\beta _{j_0}
		\leq H(\alpha )^{\frac{d(\theta +\varepsilon )}{\delta -\varepsilon (1+d)}}
		< \beta _{j_0+1}.
	\end{equation*}
	
	We first consider the case of $\alpha \neq \alpha _{j_0}$.
	Since
	\begin{equation*}
		H(\alpha )^d
		< \beta_{j_0+1}^{\frac{\delta -\varepsilon (1+d)}{\theta +\varepsilon }} 
		\leq \beta _{j_0}^{\delta -\varepsilon (1+d)},
	\end{equation*}
	we get
	\begin{equation*}
		\abs{\alpha -\alpha _{j_0}}
		\geq H(\alpha )^{-d}H(\alpha _{j_0})^{-d}
		> \beta _{j_0}^{-d-\delta +\varepsilon }
		\geq \abs{\xi -\alpha_{j_0}}
	\end{equation*}
	by Lemma \ref{lem:Liovilleineq}.
	Therefore, we obtain
	\begin{align*}
		\abs{\xi -\alpha }
		& = \abs{\alpha -\alpha_{j_0}}
		\geq H(\alpha )^{-d}H(\alpha _{j_0})^{-d}\\
		& \geq H(\alpha )^{-d}\beta_{j_0}^{-d(1+\varepsilon )}
		\geq H(\alpha )^{-d-\frac{d^2(\theta +\varepsilon )(1+\varepsilon )}{\delta -\varepsilon (1+d)}}.
	\end{align*}
	
	We next consider the case of $\alpha =\alpha_{j_0}$.
	By the assumption, we have
	\begin{equation*}
		\abs{\xi -\alpha }
		\geq \beta _{j_0}^{-d-\rho -\varepsilon }
		\geq H(\alpha )^{-(d+\rho +\varepsilon)\frac{d(\theta +\varepsilon)}{\delta -\varepsilon (1+d)}}.
	\end{equation*}
	Since $\varepsilon $ is arbitrary small, we deduce that
	\begin{equation*}
		w_d^{*}(\xi )
		\leq \max \left( d+\frac{d^2\theta}{\delta}-1, (d+\rho)\frac{d\theta}{\delta}-1 \right) 
		= (d+\rho )\frac{d\theta }{\delta }-1.
	\end{equation*}
\end{proof}

\section{Proof of Main Results}\label{sec:proof}

\begin{proof}[Proof of Theorem \ref{thm:main}]
	For simplicity of notation, we put $\xi :=\xi (r,{\bf m})$ and $\alpha_j:=\alpha_j(r,{\bf m})$.
	For an integer $j\geq 0$, we define sequences $(a(j,n))_{n\geq 1}$ and $(b(j,n))_{n\geq 1}$ in $\lF _q$ by
	\begin{equation*}
	\sum_{n=0}^{j}\alpha _n 
	= \sum_{n=1}^{\infty }a(j,n)T^{-r^n},\quad 
	\sum_{n=j+1}^{\infty }\alpha _n 
	= \sum_{n=1}^{\infty }b(j,n)T^{-r_{j+1}^n}.
	\end{equation*} 
	For integers $j\geq i\geq 0$, we put $M(i,j):=m_im_{i+1}\cdots m_j$.
	For convenience, we put $M(i,j):=1$ for integers $i>j\geq 0$.
	Then it is easy to check that
	\begin{equation}\label{eq:bjn}
	b(j,n)
	=\ell \bmod p,
	\end{equation}
	where $\ell \geq 1$ is an integer with $M(j+2,j+\ell )\mid n$ and $M(j+2,j+\ell +1)\nmid n$.
	For integers $j\geq 0$ and $k\geq 1$, we define algebraic Laurent series $\alpha (j,k)$ by
	\begin{equation*}
	\alpha (j,k) 
	= \sum_{n=0}^{j}\alpha _n+\sum_{n=1}^{k}b(j,n)T^{-r_{j+1}^n}.
	\end{equation*}
	
	In what follows, we estimate upper bounds of degree and height of $\alpha (j,k)$.
	We observe that
	\begin{equation}\label{eq:alpha}
	\alpha (j,k)^{r_j} 
	= \sum_{n=0}^{j}\alpha _n ^{r_j}+\sum_{n=1}^{k}b(j,n)T^{-r_{j+1}^nr_j}.
	\end{equation}
	By the equation \eqref{eq:equation}, for an integer $0\leq n\leq j$, we have
	\begin{align*}
	\alpha _n^{r_j}
	& = (\alpha _n-T^{-r_n})^{r_n^{M(n+1,j)-1}}
	= \alpha_n^{r_n^{M(n+1,j)-1}}-T^{-r_n^{M(n+1,j)}}\\
	& = \cdots 
	= \alpha _n -\sum_{i=1}^{M(n+1,j)}T^{-r_n^i}.
	\end{align*}
	By the definition of the sequence $(a(j,n))_{n\geq 1}$, we obtain
	\begin{equation*}
	\sum_{n=0}^{j}\alpha _n ^{r_j} 
	= \sum_{n=0}^{j}\alpha _n - \sum_{n=1}^{M(1,j)}a(j,n)T^{-r^n}.
	\end{equation*}
	Therefore, by \eqref{eq:alpha}, $\alpha (j,k)$ is a root of the polynomial
	\begin{equation*}
	X^{r_j}-X+\sum_{n=1}^{M(1,j)}a(j,n)T^{-r^n}+\sum_{n=1}^{k}b(j,n)T^{-r_{j+1}^n}-\sum_{n=1}^{k}b(j,n)T^{-r_{j+1}^nr_j}.
	\end{equation*}
	Hence, it follows from Lemma \ref{lem:height} that $\deg \alpha (j,k)\leq r_j$ and $H(\alpha (j,k))\leq q^{r_{j+1}^kr_j}$.
	
	For an integer $j\geq 0$, we denote by $K_j$ the set of all integers $k\geq 1$ with $M(j+2,j+p)\mid (k+1)$ and $M(j+2,j+p+1)\nmid (k+1)$.
	Note that $K_j$ is the infinite set.
	We put $K_j =: \{ k_1<k_2<\ldots \}$.
	We observe that for all integers $i\geq 1$,
	\begin{equation}\label{eq:ki}
	k_{i+1}-k_i\leq 2M(j+2,j+p).
	\end{equation}
	
	We show that for all integers $i\geq 1$,
	\begin{equation}\label{eq:app}
	\abs{\xi -\alpha (j,k_i)}
	= q^{-r_{j+1}^{k_i+2}}.
	\end{equation}
	We observe that
	\begin{equation*}
	\abs{\xi -\alpha (j,k_i)} 
	= \abs{\sum_{n=k_i+1}^{\infty }b(j,n)T^{-r_{j+1}^n}}.
	\end{equation*}
	By \eqref{eq:bjn}, we have $b(j,k_i+1)=0$.
	Then we deduce that $M(j+2,j+2)\mid (k_i+1)$, which implies $M(j+2,j+2)\nmid (k_i+2)$.
	Therefore, we get $b(j,k_i+2)=1$.
	Hence, we obtain \eqref{eq:app}.
	
	For integers $j\geq 0$ and $k\geq 1$, we put $\beta (j,k):=q^{r_{j+1}^kr_j}$.
	Then, by \eqref{eq:ki} and \eqref{eq:app}, we have
	\begin{equation*}
	\frac{-\log \abs{\xi -\alpha (j,k_i)}}{\log \beta (j,k_i)}
	= \frac{r_{j+1}^2}{r_j},\quad 
	\frac{\log \beta (j,k_{i+1})}{\log \beta (j,k_i)}
	\leq r_{j+p}^2
	\end{equation*}
	for all integers $i\geq 1$.
	It is trivial that $r_{j+1}^2/r_j>r_j$ for all integers $j\geq 0$.
	Applying Lemma \ref{lem:applio} with $d=r_j, \delta = \rho = r_{j+1}^2/r_j-r_j,$ and $\theta = r_{j+p}^2$, we obtain
	\begin{equation}\label{eq:wn}
	\frac{r_{j+1}^2}{r_j}-1
	\leq w_{r_j}^{*}(\xi )
	\leq \frac{r_jr_{j+1}^2r_{j+p}^2}{r_{j+1}^2-r_j^2}-1
	\end{equation}
	for all integers $j\geq 0$.
	Hence, it follows that $w_n^{*}(\xi )<+\infty $ for all integers $n\geq 1$.
	We also have
	\begin{equation*}
	w^{*}(\xi )
	\geq \limsup_{j\rightarrow \infty } \frac{w_{r_j}^{*}(\xi )}{r_j}
	\geq \limsup_{j\rightarrow \infty } \left( \frac{r_{j+1}^2}{r_j^2}-\frac{1}{r_j}\right)
	= +\infty . 
	\end{equation*}
	Thus, the Laurent series $\xi $ is $T^{*}$-number.
	Therefore, by Section \ref{sec:mainresult}, we deduce that $\xi $ is $T$-number.
\end{proof}

\begin{proof}[Proof of Theorem \ref{thm:main2}]
	Assume the notation of the proof of Theorem \ref{thm:main}.	
	By \eqref{eq:wn} and Lemma \ref{lem:tnumber}, we obtain
	\begin{equation*}
	t(\xi )
	\geq t^{*}(\xi )
	\geq \limsup_{j\rightarrow \infty } \frac{\log w_{r_j}^{*}(\xi )}{\log r_j}
	\geq \limsup_{j\rightarrow \infty }(2m_j-1).
	\end{equation*}
	Since $r_{j+1}^2/(r_{j+1}^2-r_j^2)\leq 2$, we deduce that for all integers $j\geq 1$ and $r_{j-1}+1\leq n\leq r_j$,
	\begin{equation*}
	\frac{\log w_n^{*}(\xi )}{\log n}
	\leq \frac{\log w_{r_j}^{*}(\xi )}{\log r_{j-1}}
	\leq m_j+2m_jm_{j+1}\cdots m_{j+p}+\frac{\log 2}{\log r_{j-1}}.
	\end{equation*}
	Therefore, we get \eqref{eq:main1}.
	
	By \eqref{eq:wn} and Lemma \ref{lem:hikaku}, we obtain
	\begin{equation*}
	w_{r_j}(\xi )
	\leq \frac{(r_jr_{j+1}r_{j+p})^2}{r_{j+1}^2-r_j^2}+r_j^2-2
	\leq 3r_j^2r_{j+p}^2
	\end{equation*}
	for all integers $j\geq 0$.
	Therefore, it follows that, for all integers $j\geq 1$ and $r_{j-1}+1\leq n\leq r_j$,
	\begin{equation*}
	\frac{\log w_n(\xi )}{\log n}
	\leq 2m_j+2m_jm_{j+1}\cdots m_{j+p}+\frac{\log 3}{\log r_{j-1}},
	\end{equation*}
	which implies \eqref{eq:main2}.
	
	Assume that $m_j\geq 3$ for all sufficiently large $j\geq 1$.
	In the same way to the proof of \eqref{eq:app}, it follows that for all integers $j\geq 0$ and $k\geq 1$,
	\begin{equation*}
	q^{-r_{j+1}^{k+2}}
	\leq \abs{\xi -\alpha (j,k)}
	\leq q^{-r_{j+1}^{k+1}}.
	\end{equation*}
	Therefore, we have
	\begin{equation*}
		\frac{r_{j+1}}{r_j}
		\leq \frac{-\log \abs{\xi -\alpha (j,k)}}{\log \beta (j,k)}
		\leq \frac{r_{j+1}^2}{r_j},\quad 
		\frac{\log \beta (j,k+1)}{\log \beta (j,k)}=r_{j+1}
	\end{equation*}
	for all $j\geq 0$ and $k\geq 1$.
	By the assumption, we obtain $r_{j+1}/r_j>r_j$ for all sufficiently large $j\geq 1$.
	Applying Lemma \ref{lem:applio}, we have
	\begin{equation}
	w_{r_j}^{*}(\xi )
	\leq \frac{r_jr_{j+1}^3}{r_{j+1}-r_j^2}-1
	\end{equation}
	for all sufficiently large $j\geq 1$.
	In the same way to the proof of \eqref{eq:main1} and \eqref{eq:main2}, we derive \eqref{eq:main3} and \eqref{eq:main4}.
\end{proof}

\begin{proof}[Proof of Theorem \ref{thm:intro}]
	Let $r$ be a power of $p$ and ${\bf m}=(m_j)_{j\geq 0}$ be an integer sequence with $m_0=1, m_j\geq 2$ for all integers $j\geq 1$.
	Let ${\bf a}=(a_j)_{j\geq 0}$ be an integer sequence with $a_j\in \{ 0,1\}$ for all integers $j\geq 0$.
	Assume that $a_j=1$ for infinitely many $j\geq 0$.
	We put
	\begin{equation*}
	\xi _{{\bf a}}(r,{\bf m}):=\sum_{j=0}^{\infty }a_j\alpha_j(r,{\bf m}).
	\end{equation*}
	Then there exist $r'$ and ${\bf m}'=(m_j')_{j\geq 0}$ such that $r'$ is a power of $p$, ${\bf m}'=(m_j')_{j\geq 0}$ is an integer sequence with $m_0'=1, m_j'\geq 2$ for all integers $j\geq 1$, and $\xi _{{\bf a}}(r,{\bf m})=\xi (r',{\bf m}')$.
	Therefore, the Laurent series $\xi _{{\bf a}}(r,{\bf m})$ is $T$-number.
	Let ${\bf a}'=(a_j')_{j\geq 0}$ be an integer sequence with ${\bf a} \neq {\bf a}', a_j' \in \{ 0,1\}$ for all integers $j\geq 0$, and $a_i'=1$ for infinitely many $i$.
	We define $j_0\geq 0$ by $a_j=a_j'$ for all $0\leq j<j_0$ and $a_{j_0}\neq a_{j_0}'$.
	Then we have
	\begin{align*}
		\abs{\xi_{{\bf a}}(r,{\bf m})-\xi_{{\bf a}'}(r,{\bf m})}
		& = \abs{(a_{j_0}-a_{j_0}')\alpha_{j_0}(r,{\bf m}) + \sum_{j=j_0+1}^{\infty} (a_j-a_j')\alpha_j(r,{\bf m})}\\
		& = \abs{\alpha_{j_0}(r,{\bf m})} = q^{-r_{j_0}} \neq 0,
	\end{align*}
	which implies $\xi_{{\bf a}}(r,{\bf m})\neq \xi_{{\bf a}'}(r,{\bf m})$.
	Since there are uncountably many choices of such sequences ${\bf a}$, the proof is complete.
\end{proof}

\appendix
\section{Basic properties of Mahler's classification}\label{sec:app}

\begin{lem}\label{lem:MKlower}
	Let $n\geq 1$ be an integer and $\xi \in \lF _q((T^{-1}))$ be not algebraic of degree at most $n$.
	Then we have $w_n (\xi ) \geq n$ and $w_n ^{*}(\xi ) \geq (n+1)/2$.
\end{lem}

\begin{proof}
	The former estimate follows from an analogue of Minkowski's theorem for Laurent series over a finite field \cite{Mahler41} and the later estimates are Satz.1 of \cite{Guntermann96}.
\end{proof}

The following lemma is Theorem 5.2 in \cite{Ooto17}.

\begin{lem}\label{lem:alg}
	Let $n\geq 1$ be an integer and $\xi \in \lF _q((T^{-1}))$ be algebraic of degree $d$.
	Then we have $w_n (\xi ), w_n ^{*}(\xi ) \leq d-1$.
\end{lem}

From Lemmas \ref{lem:MKlower} and \ref{lem:alg}, we see the following proposition.

\begin{prop}\label{prop:A-number}
	Let $\xi $ be in $\lF _q((T^{-1}))$.
	Then the following conditions are equivalent:
	\begin{enumerate}
		\item $\xi $ is an $A$-number, 
		\item $\xi $ is an $A^{*}$-number, 
		\item $\xi $ is an algebraic Laurent series.
	\end{enumerate}
\end{prop}

Let $\xi $ be in $\lF_q((T^{-1}))$ and $n,H\geq 1$ be integers.
We put
\begin{equation*}
	w_n(\xi ,H) := \min \{ \abs{P(\xi )}\mid P(X)\in (\lF_q[T])[X], H(P)\leq H, \deg _X P\leq n, P(\xi )\neq 0 \} .
\end{equation*}
It is easy to check that
\begin{equation*}
	w_n(\xi ) = \limsup_{H\rightarrow \infty }\frac{-\log w_n(\xi ,H)}{\log H}.
\end{equation*}

\begin{prop}\label{prop:dependent}
	Let $\xi ,\eta \in \lF_q ((T^{-1}))$ be transcendental Laurent series.
	If $\xi $ and $\eta $ are algebraically dependent over $\lF_q(T)$, then $\xi $ and $\eta $ are in the same Mahler's class.
\end{prop}

\begin{proof}
	For an integer $H\geq 1$, we take a polynomial $P(X) \in (\lF _q[T])[X]$ with $H(P)\leq H, \deg _X P\leq n$, and $|P(\xi )|=w_n(\xi ,H)$.
	There exists $F(X,Y)\in (\lF _q[T])[X,Y]$ which is an irreducible primitive polynomial in $X$ and $Y$ such that $F(\xi ,\eta )=0$.
	We write 
	\begin{equation*}
	F(X,Y)=\sum _{i=0}^{M}\sum _{j=0}^{N}a_{i j}X^i Y^j=\sum _{i=0}^{M}B_i(Y)X^i,
	\end{equation*}
	where $a_{i j}\in \lF _q[T], B_i(Y)\in (\lF _q[T])[Y]$, and $B_M(Y) \neq 0$.
	Since there exists $y\in \lF _q(T)$ such that $P(X)$ and $F(X,y)$ have no common root, it follows that the resultant $R(Y)=\Res _X(P(X),F(X,Y))$ is non-zero polynomial in $(\lF _q[T])[Y]$.
	Then we obtain $\deg _Y R(Y)\leq nN$ and there exists an integer $c_1\geq 1$ such that $H(R)\leq c_1 H^M$.
	By the basic property of resultants (see e.g.\ \cite[p.199--200]{Lang02}), there exist polynomials $g(X,Y), h(X,Y) \in (\lF _q[T])[X,Y]$ and an integer $c_2\geq 1$ such that $R(Y)=P(X)g(X,Y)+F(X,Y)h(X,Y)$ and all of the absolute values of the coefficients of $g(X,Y)$ are less than or equal to $c_2H^{M-1}$.
	Then we have $R(\eta )=P(\xi )g(\xi ,\eta )$ and $|g(\xi ,\eta )|\leq c_3H^{M-1}$ for some integer $c_3\geq 1$.
	Therefore, we obtain $w_n(\xi )\leq M-1+M w_{n N}(\eta )$ and $w(\xi )\leq MN w(\eta )$.
	
	We change a role of $\xi $ and $\eta $, which implies $w_n(\eta )\leq N-1+N w_{n M}(\xi )$ and $w(\eta )\leq MN w(\xi )$.
	This completes the proof.
\end{proof}

\subsection*{Acknowledgements}
The author would like to express his gratitude to Prof.~Shigeki Akiyama and Prof.~Hajime Kaneko for helpful comments.
The author would like to thank Prof.~Yann Bugeaud for several comments.
The author would like to thank the referee.

\end{document}